\documentclass[12pt]{amsart}

\usepackage[mathscr]{eucal}
\usepackage{amsmath,amssymb,amscd,amsthm}%,latexsym
\usepackage{color}
\usepackage{amsmath,amsthm,amssymb,graphicx}
\usepackage{epsfig}
\newtheorem{theorem}{Theorem}
\newtheorem{lemma}{Lemma}

\newtheorem{definition}{Definition}

\theoremstyle{definition}
\newtheorem{remark}{Remark}
\theoremstyle{plane}

\def \beq{ \begin{equation} }
\def \eeq{\end{equation}}

\newcommand\co{\mbox{%
\raisebox{.8ex}{c}%
\kern-.175em\raisebox{.2ex}{/}%
\kern-.18em\raisebox{-.2ex}{o}%
}~}
\usepackage[top=3.8cm, bottom=3.8cm, left=3.5cm, right=3.5cm]{geometry}
\title{All the Lagrangian relative equilibria of the curved 3-body problem have equal masses}
\begin{document}
\maketitle
\markboth{Florin Diacu and Sergiu Popa}{All the Lagrangian relative equilibria of the curved 3-body problem have equal masses}
\author{\begin{center}
{\bf Florin Diacu}$^{1,2}$ and {\bf Sergiu Popa} (\co \!$^2$)\\
\smallskip
{\footnotesize $^1$Pacific Institute for the Mathematical Sciences\\
and\\
$^2$Department of Mathematics and Statistics\\
University of Victoria\\
P.O.~Box 1700 STN CSC\\
Victoria, BC, Canada, V8W 2Y2\\
diacu@uvic.ca\\
}\end{center}

}

\vskip0.5cm

\begin{center}
\today
\end{center}

\begin{abstract}
We consider the 3-body problem in 3-dimensional spaces of nonzero constant Gaussian curvature and study the relationship between the masses of the Lagrangian relative equilibria, which are orbits that form a rigidly rotating equilateral triangle at all times. There are three classes of Lagrangian relative equilibria in 3-dimensional spaces of constant nonzero curvature: positive elliptic and positive elliptic-elliptic, on 3-spheres, and negative elliptic, on hyperbolic 3-spheres. We prove that all these Lagrangian relative equilibria exist only for equal values of the masses.
\end{abstract}
 
%%%%%%%%
%%%%%%%%
%%%%%%%%
\section{Introduction}

The idea of studying the gravitational motion of point masses in spaces of constant nonzero curvature belongs to Bolyai and Lobachevsky, who independently proposed a Kepler problem (the motion of a body around a fixed attractive centre) that would naturally extend Newton's law to hyperbolic space, \cite{Bolyai}, \cite{Lobachevsky}. Their idea was further developed by Schering, \cite{Schering}, Killing, \cite{Killing}, Liebmann, \cite{Liebmann1}, \cite{Liebmann2}, to both hyperbolic and elliptic space, and more recently by Cari\~nena, Ra\~nada, and Santander, \cite{Carinena}, as well as by the Russian school of celestial mechanics, which also studied the 2-body problem in spaces of constant curvature (see \cite{Kozlov} and its references).

In Euclidean space, the Kepler problem and the 2-body problem are equivalent. This is no longer true in curved space, due to the difficulty of giving a meaning to the centre of mass. Moreover, Shchepetilov showed that, unlike the curved Kepler problem, the curved 2-body problem is not integrable, \cite{Shchepetilov}, a fact which reveals that the latter problem is far from trivial. The level of difficulty increases for the curved $N$-body problem with $N\ge 3$,
whose systematic study started recently.

In the past few years, the efforts spent in this direction generated a body of work that opened the door to an interesting research topic,
\cite{Diacu1}, \cite{Diacu2}, \cite{Diacu3}, \cite{Diacu4},  \cite{Diacu15}, \cite{Diacu14}, \cite{Diacu9}, \cite{Diacu5}, \cite{Diacu6}, \cite{Diacu7}, \cite{Diacu8}, \cite{Diacu10}, \cite{Martinez}, \cite{Perez}, \cite{Tibboel1}, \cite{Tibboel2}, \cite{Tibboel3}, \cite{Zhu}. The results obtained so far have not only established it as a new branch of celestial mechanics, but also showed that it helps better understand the classical Newtonian case, \cite{Diacu14}. Among the potential applications of the curved $N$-body problem is that of establishing the shape of the universe, i.e., deciding whether the macrocosm has negative, zero, or positive curvature. By proving the existence of orbits that show up only in one kind of universe, but not in the other two, and finding them through astronomical observations, we could ascertain the geometry of the physical space. 

Our paper makes a small step in this direction. We consider Lagrangian relative equilibria of the 3-body problem in spaces of constant Gaussian curvature. Lagrange was the first to discover such solutions in the Euclidean space. He showed that any three masses, $m_0, m_1, m_2>0$, lying at the vertices of an equilateral triangle can rotate on concentric circles around the common centre of mass maintaining all the time the size and shape of the initial triangle, \cite{Wintner}. We recently proved that on 2-spheres and hyperbolic 2-spheres, Lagrangian relative equilibria exist only if $m_0=m_1=m_2$, \cite{Diacu7}. However, whether all three masses must be equal remained an open question in 3-dimensional spaces of nonzero curvature. If we proved that all Lagrangian and near-Lagrangian orbits must have equal masses for nonzero curvature, then the universe would be flat, since orbits of nonequal masses like those discovered by Lagrange have been observed in our solar system. Here we consider only the Lagrangian case. The near-Lagrangian orbits are much harder to deal with, and their study would be a next step in the desired direction.

While in 2- and 3-dimensional Euclidean space there is just one class of isometric rotations, namely those corresponding to the Lie groups $SO(2)$ and $SO(3)$, respectively, things get more involved for curvatures $\kappa\ne 0$. On 2-spheres we have only the isometric rotations corresponding to $SO(3)$, called elliptic, but on hyperbolic 2-spheres the corresponding Lie group is $Lor(\mathbb R^{2,1})$ (the Lorentz group over the 3-dimensional Minkowski space $\mathbb R^{2,1}$), which involves elliptic, hyperbolic, and parabolic rotations. We showed that parabolic rotations don't lead to any relative equilibria, but elliptic and hyperbolic rotations do. However, Lagrangian relative equilibria given by hyperbolic rotations don't exist. So we are left only with the elliptic Lagrangian orbits for both positive and negative curvature. As mentioned earlier, they occur solely when the masses are equal (see \cite{Diacu3} or \cite{Diacu4} for proofs of all these statements). 

In 3-dimensional spaces of curvature $\kappa\ne 0$, things get a bit more complicated. The rotation group of the 3-sphere is $SO(4)$, which involves one or two rotations. We call elliptic the elements involving a single rotation and elliptic-elliptic the elements having two rotations. The Lagrangian relative equilibria corresponding to both these isometries exist for equal masses, as proved in \cite{Diacu3} and \cite{Diacu4}. 

The isometric rotations of the hyperbolic 3-sphere are given by the Lorentz group $Lor(\mathbb R^{3,1})$, where $\mathbb R^{3,1}$ is the 4-dimensional Minkowski space in which the hyperbolic 3-sphere is embedded. This group involves elliptic, hyperbolic, elliptic-hyperbolic, and parabolic rotations. Again, parabolic rotations don't lead to any relative equilibria, and there are no Lagrangian orbits given by hyperbolic and elliptic-hyperbolic rotations. But elliptic Lagrangian relative equilibria of equal masses exist for negative curvature as well, as shown in \cite{Diacu3} and \cite{Diacu4}. So, in 3-dimensional spaces, equal-mass positive elliptic and positive elliptic-elliptic (for $\kappa>0$) and  negative elliptic (for $\kappa<0$) Lagrangian relative equilibria exist. 

Given the above remarks, the question asked earlier can now be better rephrased. Do there exist 3-dimensional: (i) positive elliptic, (ii) positive elliptic-elliptic, and (iii) negative elliptic Lagrangian relative equilibria for nonequal masses? We will further prove that the answer is negative: in all cases, (i), (ii), and (iii), Lagrangian relative equilibria exist only if the masses are equal.

The rest of this paper is organized as follows. In Section 2 we introduce the equations of motion and their first integrals.
In Section 3 we show that the positive elliptic and negative elliptic Lagrangian relative equilibria exist only if all three masses are equal. For both these kinds of orbits, we use the same idea to prove our result. Finally, in Section 4, employing a different method than in Section 3, we show that positive elliptic-elliptic Lagrangian relative equilibria exist also only for equal masses. These results show that the non-equal mass Lagrangian orbits of Euclidean space, also observed in our solar system, depict a rather exceptional phenomenon, which does not occur in spaces of nonzero constant curvature. 

%%%%%%%%
%%%%%%%%
%%%%%%%%
\section{Equations of motion}

In our previous work, we proved that when studying qualitative properties of the curved $N$-body problem, $N\ge 2$, it is not necessary to take into account the value of the Gaussian curvature $\kappa$, but only its sign. Therefore, it is enough to work with the values $\kappa=1$ and $\kappa=-1$, which means that we can restrict our considerations to the unit sphere and unit hyperbolic sphere. So we consider three point particles (bodies) of masses $m_0, m_1, m_2>0$ moving in
$\mathbb{S}^3$ (embedded in $\mathbb R^4$) or in $\mathbb H^3$ (embedded in the Minkowski space $\mathbb R^{3,1}$),
where
$$
{\mathbb S}^3=\{(w,x,y,z)\ | \ w^2+x^2+y^2+z^2=1\},
$$
$$
{\mathbb H}^3=\{(w,x,y,z)\ | \ w^2+x^2+y^2-z^2=-1, \ z>0\}.
$$
Then the configuration of the system is described by the vector
$$
{\bf q}=({\bf q}_0,{\bf q}_1,{\bf q}_2),
$$
where ${\bf q}_i=(w_i,x_i,y_i,z_i),\  i=0,1,2,$ denote the position vectors of the bodies.  The equations of motion (see \cite{Diacu3}, \cite{Diacu4}, or \cite{Diacu7} for their derivation using constrained Lagrangian dynamics) are given by the system
\begin{equation}
\label{both}
\ddot{\bf q}_i=\sum_{j=0,j\ne i}^2\frac{m_j({\bf q}_j-\sigma q_{ij}{\bf q}_i)}{(\sigma-\sigma q_{ij}^2)^{3/2}}-{\sigma}(\dot{\bf q}_i\cdot\dot{\bf q}_i){\bf q}_i, \ \ i=0,1,2,
\end{equation}
with initial-condition constraints
\begin{equation}
q_{ii}(0)=\sigma, \ \ \ ({\bf  q}_i\cdot\dot{\bf q}_i)(0)=0, \ \ i=0,1,2,
\end{equation}
where  $q_{ij}={\bf q}_i\cdot{\bf q}_j$, the dot is the standard inner product of signature $(+,+,+,+)$ in $\mathbb S^3$, but the Lorentz inner product of signature $(+,+,+,-)$ in $\mathbb H^3$, and
$\sigma=\pm 1$, depending on whether the curvature is positive or negative. 

From Noether's theorem, system \eqref{both} has the energy integral,
$$
T({\bf q},\dot{\bf q})-U({\bf q})=h,
$$
where 
$$U({\bf q})=\sigma\sum_{0\le i<j\le 2}\frac{m_im_jq_{ij}}{(\sigma-\sigma q_{ij}^2)^{1/2}}$$ 
is the force function ($-U$ is the potential);
$$
T({\bf q},\dot{\bf q})=\frac{\sigma}{2}\sum_{i=0}^2 m_iq_{ii}
\dot{\bf q}_i\cdot\dot{\bf q}_i
$$
is the kinetic energy, with $h$ representing an integration constant; and the integrals of the total angular momentum,
$$
\sum_{i=0}^2m_i{\bf q}_i\wedge\dot{\bf q}_i={\bf c},
$$
where $\wedge$ is the wedge product and ${\bf c}=(c_{wx},c_{wy},c_{wz},c_{xy},c_{xz},c_{yz})$ denotes a constant integration vector, each component measuring the rotation of the system about the origin of the frame relative to the plane corresponding to the bottom indices. On components, the 6 integrals of the total angular momentum are given by the equations
\begin{align*}\label{angularmomentum}
c_{wx}&=\sum_{i=0}^2m_i(w_i\dot{x}_i-\dot{w}_ix_i), & c_{wy}&=\sum_{i=0}^2m_i(w_i\dot{y}_i-\dot{w}_iy_i), &\\
c_{wz}&=\sum_{i=0}^2m_i(w_i\dot{z}_i-\dot{w}_iz_i), & 
c_{xy}&=\sum_{i=0}^2m_i(x_i\dot{y}_i-\dot{x}_iy_i), &\\
c_{xz}&=\sum_{i=0}^2m_i(x_i\dot{z}_i-\dot{x}_iz_i), &  
c_{yz}&=\sum_{i=0}^2m_i(y_i\dot{z}_i-\dot{y}_iz_i).
\end{align*}
System \eqref{both} has no integrals of the centre of mass and linear momentum, so most of the techniques used in the Euclidean case cannot be applied in spaces of nonzero constant curvature, \cite{Diacu4}, \cite{Diacu14}. The lack of these integrals also means that the equations of motion of the curved problem have fewer symmetries than in the Euclidean case, a fact that makes the problem more difficult,  \cite{Diacu15}, \cite{Diacu14}.

%%%%%%%%
%%%%%%%%
%%%%%%%%

\section{Positive and negative elliptic Lagrangian relative equilibria} 

In this section we will show that both the positive elliptic and the negative elliptic Lagrangian relative equilibria exist only if all the three masses are equal. Let's start by the recalling the definition of these solutions.

\begin{definition}\label{elliptic}
Consider the masses $m_0, m_1, m_2>0$ and a solution of system \eqref{both} of the form
$$
{\bf q}=({\bf q}_0, {\bf q}_1, {\bf q}_2), \ \ {\bf q}_i=(w_i,x_i,y_i,z_i),\ \ i=0,1,2,
$$
with
$$
w_0(t)=r_0\cos(\omega t+a_0),\ \ x_0(t)=r_0\sin(\omega t+a_0), \ \ y_0(t)=y_0,\ \ z_0(t)=z_0,
$$
$$
w_1(t)=r_1\cos(\omega t+a_1),\  \ x_1(t)=r_1\sin(\omega t+a_1), \ \ y_1(t)=y_1,\ \ z_0(t)=z_1,
$$
$$
w_2(t)=r_2\cos(\omega t+a_2),\ \ x_2(t)=r_2\sin(\omega t+a_2), \  \ y_2(t)=y_2,\ \ z_2(t)=z_2,
$$
where $\omega\ne 0, r_0,r_1,r_2>0$, and $y_0,y_1,y_2,z_1,z_2,z_3, a_0, a_1, a_2$ are
constants with 
$$
r_i^2+y_i^2+\sigma z_i^2=\sigma,\ i=0,1,2,
$$
as well as $z_0,z_1,z_2>0$ if $\sigma=-1$. Moreover, assume that
$$
q_{01}=q_{02}=q_{12},
$$
conditions which imply that the triangle is equilateral.
Then, for $\sigma=1$, the above solution is called a positive  elliptic Lagrangian relative equilibrium. For $\sigma=-1$, it is called a negative elliptic Lagrangian relative equilibrium.
\end{definition}

\begin{remark}
Without loss of generality, we can assume that $a_0=b_0=0$.
\end{remark}

Before focusing on the main result of this section, let us prove a property that concerns the integrals of the total angular momentum.

\begin{lemma}\label{lemma1}
Consider a positive or negative elliptic Lagrangian relative equilibrium of system \eqref{both}. Then $c_{wy}=c_{wz}=0$,
where $c_{wy}$ and $c_{wz}$ are the constants of the total angular momentum corresponding to the $wy$- and the $wz$-plane, respectively.
\end{lemma}
%%%%%%%%%
\begin{proof}
Without loss of generality, we can choose $a_0=0$ in Definition \ref{elliptic}. Notice that for such solutions we have
$$
c_{wy}=\omega m_0r_0y_0\sin\omega t+\omega m_1r_1y_1\sin(\omega t+a_1)+\omega m_2r_2y_2\sin(\omega t+a_2),
$$ 
$$
c_{wz}=\omega m_0r_0z_0\sin\omega t+\omega m_1r_1z_1\sin(\omega t+a_1)+\omega m_2r_2z_2\sin(\omega t+a_2).
$$
Then, on one hand, it follows that, for $t=0$, 
\begin{equation}\label{t=0,cwy}
c_{wy}=\omega m_1r_1y_1\sin a_1+\omega m_2r_2y_2\sin a_2,
\end{equation}
\begin{equation}\label{t=0,cwz}
c_{wz}=\omega m_1r_1z_1\sin a_1+\omega m_2r_2z_2\sin a_2,
\end{equation}
but, on the other hand, for $t=\pi/\omega$,
$$
c_{wy}=-\omega m_1r_1y_1\sin a_1-\omega m_2r_2y_2\sin a_2,
$$
$$
c_{wz}=-\omega m_1r_1z_1\sin a_1-\omega m_2r_2z_2\sin a_2.
$$
Therefore $c_{wy}=-c_{wy}$ and $c_{wz}=-c_{wz}$, so the conclusion follows.
\end{proof}

\begin{remark}
We can similarly show that $c_{xy}=c_{xz}=0$. It is obvious that $c_{yz}=0$. These facts, however, will be of no further use in this paper.  
\end{remark}

We can now prove the following result.

%%%%%%%%
\begin{theorem}\label{theorem-elliptic}
System \eqref{both} has positive elliptic Lagrangian relative equilibria, for $\sigma=1$, and negative elliptic Lagrangian relative equilibria, for $\sigma=-1$, if and only if all three masses are equal.
\end{theorem}
%%%%%%%%
\begin{proof}
The existence of equal-mass positive elliptic and negative elliptic Lagrangian relative equilibria was established in \cite{Diacu3} and \cite{Diacu4}. We will further show that 
all positive elliptic and negative elliptic Lagrangian relative equilibria must have equal masses.

Consider a positive or negative elliptic Lagrangian relative equilibrium of system \eqref{both}. Lemma \ref{lemma1} implies that $c_{wy}=0$, so by \eqref{t=0,cwy} we have
\begin{equation}\label{first}
m_1r_1y_1\sin a_1=-m_2r_2y_2\sin a_2.
\end{equation}
If we further consider the equation of system \eqref{both} corresponding to $\ddot y_1$ at $t=0$, we obtain that
\begin{equation}\label{second}
m_1r_1\sin a_1=-m_2r_2\sin a_2.
\end{equation}
From \eqref{first} and \eqref{second} we can conclude that $y_1=y_2$. 

Lemma \ref{lemma1} also implies that $c_{wz}=0$, so by \eqref{t=0,cwz} we have
\begin{equation}\label{third}
m_1r_1z_1\sin a_1=-m_2r_2z_2\sin a_2.
\end{equation}
From \eqref{second} and \eqref{third}, we can conclude that $z_1=z_2$. But from Definition \ref{elliptic}, we have that $r_1^2+y_1^2+\sigma z_1^2=\sigma, r_2^2+y_2^2+\sigma z_2^2=\sigma$, and $r_1, r_2>0$. Therefore $r_1=r_2$. Proceeding similarly we can draw the conclusion that
$$
r_1=r_2=r_3=:r,\ \ y_1=y_2=y_3=: y,\ \ z_1=z_2=z_3=: z,
$$
so any positive or negative elliptic Lagrangian relative equilibrium must have the form
\begin{align*}
w_0(t)&=r\cos\omega t, & x_0(t)&=r\sin\omega t, & y_0(t)&=y, & z_0(t)&=z,\\
w_1(t)&=r\cos(\omega t+a_1), & x_1(t)&=r\sin(\omega t+a_1), & y_1(t)&=y, & z_1(t)&=z,\\ 
w_2(t)&=r\cos(\omega t+a_2),& x_2(t)&=r\sin(\omega t+a_2), & y_2(t)&=y, & z_2(t)&=z, 
\end{align*}
with $r^2+y^2+\sigma z^2=\sigma$.

From the above form of the solution we see that if we fix an admissible value of the constant $z$ (in other words we look at the 3-dimensional Euclidean space $wxy$), the plane of the rotating equilateral triangle whose vertices carry the masses $m_0, m_1,$ and $m_2$ must be parallel with the $wx$-plane. Consequently $a_1=2\pi/3$ and $a_2=4\pi/3$, since we already assumed $a_0=0$. Then $\sin a_1=-\sin a_2$, so from \eqref{second} we obtain that $m_1=m_2$. The final conclusion, namely that
$
m_0=m_1=m_2,
$  
follows similarly. This conclusion completes the proof.
\end{proof}

%%%%%%%%%%
%%%%%%%%%%
%%%%%%%%%%

\section{Positive elliptic-elliptic Lagrangian relative equilibria}

In this section we will prove that all positive elliptic-elliptic Lagrangian relative equilibria exist only if all the three masses are equal. We start by the recalling the definition of these solutions.

%%%%%%%%%%%
\begin{definition}\label{elliptic-elliptic}
Consider the masses $m_0, m_1, m_2>0$ and a solution of system \eqref{both} of the form
$$
{\bf q}=({\bf q}_0, {\bf q}_1, {\bf q}_2), \ \ {\bf q}_i=(w_i,x_i,y_i,z_i),\ \ i=0,1,2,
$$
with
\begin{align*}
w_0(t)&=r_0\cos(\alpha t+a_0),& x_0(t)&=r_0\sin(\alpha t+a_0),\displaybreak[0]\\
y_0(t)&=\rho_0\cos(\beta t+b_0),& z_0(t)&=\rho_0\sin(\beta t+b_0),\displaybreak[0]\\
w_1(t)&=r_1\cos(\alpha t+a_1),& x_1(t)&=r_1\sin(\alpha t+a_1),\displaybreak[0]\\
y_1(t)&=\rho_1\cos(\beta t+b_1),& z_1(t)&=\rho_1\sin(\beta t+b_1),\displaybreak[0]\\
w_2(t)&=r_2\cos(\alpha t+a_2),& x_2(t)&=r_2\sin(\alpha t+a_2),\displaybreak[0]\\ 
y_2(t)&=\rho_2\cos(\beta t+b_2),& z_2(t)&=\rho_2\sin(\beta t+b_2),
\end{align*}
where $\alpha, \beta\ne 0, r_0,r_1, r_2, \rho_0, \rho_1, \rho_2 \ge 0$, and $a_0, a_1, a_2, b_0, b_1, b_2$ are
constants, with 
$$
r_i^2+y_i^2+z_i^2=1,\ i=0,1,2.
$$
Moreover, assume that
$$
q_{01}=q_{02}=q_{12},
$$
conditions which imply that the triangle is equilateral.
Then, the above solution of system \eqref{both}, which occurs only for $\sigma=1$, is called a positive elliptic-elliptic Lagrangian relative equilibrium. 
\end{definition}
%%%%%%%%%%%%%%
\begin{remark}
Without loss of generality, we can assume that $a_0=b_0=0$.
\end{remark}

To prove the main result of this section we will later need the following criterion, which appears in \cite{Diacu4}, pp.\ 71-72,
as well as in \cite{Diacu3}, pp.\ 44-45, for any $N\ge 2$. Since we are interested here only in the curved 3-body problem, we present this criterion in the case $N=3$.

\begin{theorem}
Consider three point particles of masses $m_0,m_1, m_2>0$,  moving in ${\mathbb S}^3$. Then, for $\alpha,\beta\ne 0$, system \eqref{both} with $\sigma=1$ admits a solution as in the above definition, but generated from a fixed-point configuration (i.e.\ a positive elliptic-elliptic relative equilibrium obtained from initial positions that would form a fixed point for zero initial velocities), if and only if there are constants $r_i, \rho_i, a_i, b_i,\ i=0,1,2$, such that the twelve relationships below are satisfied:
\beq
\sum_{\stackrel{j=0}{j\ne i}}^2\frac{m_j(r_j\cos a_j-\omega_{ij}r_i\cos a_i)}{(1-\omega_{ij}^2)^{3/2}}=0,
\label{w10}
\eeq
\beq
\sum_{\stackrel{j=0}{j\ne 2}}^3\frac{m_j(r_j\sin a_j-\omega_{ij}r_i\sin a_i)}
{(1-\omega_{ij}^2)^{3/2}}=0,
\label{x10}
\eeq
\beq
\sum_{\stackrel{j=0}{j\ne i}}^2\frac{m_j(\rho_j\cos b_j-\omega_{ij}\rho_i\cos b_i)}{(1-\omega_{ij}^2)^{3/2}}=0,
\label{y10}
\eeq
\beq
\sum_{\stackrel{j=0}{j\ne i}}^2\frac{m_j(\rho_j\sin b_j-\omega_{ij}\rho_i\sin b_i)}
{(1-\omega_{ij}^2)^{3/2}}=0, 
\label{z10}
\eeq
$i=0,1,2$, where $\omega_{ij}=r_ir_j\cos(a_i-a_j)+\rho_i\rho_j\cos(b_i-b_j), \ i,j=0,1,2,\ i\ne j$, and, additionally, one of the following two properties takes place:

(i) there is a proper subset ${\mathcal J}\subset\{0,1,2\}$ such that
$r_i=0$ for all $i\in{\mathcal J}$ and $\rho_j=0$ for all $j\in\{0,1,2\}\setminus{\mathcal J}$,

(ii) the frequencies\index{frequency}  $\alpha,\beta\ne 0$ satisfy the condition $|\alpha|=|\beta|$.

\end{theorem}

We are now in a position to prove the following result.

%%%%%%%%%%
\begin{theorem}\label{criterion}
System \eqref{both}, with $\sigma=1$, has positive elliptic-elliptic Lagrangian relative equilibria if and only if all three masses are equal. Moreover, such solutions are always generated from fixed point configurations in $\mathbb S^3$.
\end{theorem}
%%%%%%%%%%
\begin{proof}  The existence of equal-mass positive elliptic-elliptic Lagrangian relative equilibria was established in \cite{Diacu3} and \cite{Diacu4}. We further show that 
all positive elliptic-elliptic Lagrangian relative equilibria must have equal masses and that they are always generated from fixed point configurations of equilateral triangles lying on great circles of great spheres.

Substitute a candidate solution as in Definition \ref{elliptic-elliptic} with $a_0=b_0=0$ in system \eqref{both} with $\sigma=1$. Some long but straightforward computations lead us to the following conditions that must be satisfied if the solutions exists:
\beq
\sum_{\stackrel{j=0}{j\ne i}}^2\frac{m_j(r_j\cos a_j-Ar_i\cos a_i)}{(1-A^2)^{3/2}}=r_i\rho_i^2(\beta^2-\alpha^2)\cos a_i,
\label{one}
\eeq
\beq
\sum_{\stackrel{j=0}{j\ne 2}}^3\frac{m_j(r_j\sin a_j-Ar_i\sin a_i)}
{(1-A^2)^{3/2}}=r_i\rho_i^2(\beta^2-\alpha^2)\sin a_i,
\label{two}
\eeq
\beq
\sum_{\stackrel{j=0}{j\ne i}}^2\frac{m_j(\rho_j\cos b_j-A\rho_i\cos b_i)}{(1-A^2)^{3/2}}=\rho_ir_i^2(\alpha^2-\beta^2)\cos b_i,
\label{three}
\eeq
\beq
\sum_{\stackrel{j=0}{j\ne i}}^2\frac{m_j(\rho_j\sin b_j-A\rho_i\sin b_i)}
{(1-A^2)^{3/2}}=\rho_ir_i^2(\alpha^2-\beta^2)\sin b_i,
\label{four} 
\eeq
$i=0,1,2$, where $A:=q_{01}=q_{02}=q_{12}$.

Since we assumed $a_0=0$, equation \eqref{two} corresponding to $i=0$ becomes
\begin{equation}\label{two-0}
m_1r_1\sin a_1=-m_2r_2\sin a_2,
\end{equation}
and since we assumed $b_0=0$, equation \eqref{four} corresponding to $i=0$ takes the form
\begin{equation}\label{four-0}
m_1\rho_1\sin b_1=-m_2\rho_2\sin b_2.
\end{equation}
Equations \eqref{two} and \eqref{four} corresponding to $i=1$ can be respectively written as
\begin{equation}\label{two-1}
r_1(m_0A+m_1+m_2A)\sin a_1=r_1\rho_1^2(\alpha^2-\beta^2)(1-A^2)^{3/2}\sin a_1,
\end{equation}
%%%%%%%%%
\begin{equation}\label{four-1}
\rho_1(m_0A+m_1+m_2A)\sin b_1=\rho_1r_1^2(\beta^2-\alpha^2)(1-A^2)^{3/2}\sin b_1.
\end{equation}

Notice that it is possible to have $r_1\sin a_1=0$ or
$\rho_1\sin b_1=0$. Assume $r_1=0$. Then, from equation \eqref{two-0}, either $r_2=0$ or $\sin a_2=0$. It is easy to see that in either case we are led to an equal-mass Lagrangian solution with$r_0=r_1=r_2=0$ (which implies that $\rho_0=\rho_1=\rho_2=1$ and $b_1=2\pi/3, b_2=4\pi/3$), i.e., an equilateral triangle that rotates on a great circle of a great sphere. If $\rho_1\sin b_1=0$, we can reach a similar conclusion in the same way. 

Let us therefore assume that $r_1\ne 0, \rho_1\ne 0$ as well as that $\sin a_1\ne 0$ and $\sin b_1\ne 0$. Then equations \eqref{two-1} and
\eqref{four-1} reduce to
%%%%%%%%%%
\begin{equation}\label{two-1new}
m_0A+m_1+m_2A=\rho_1^2(\alpha^2-\beta^2)(1-A^2)^{3/2},
\end{equation}
%%%%%%%%%
\begin{equation}\label{four-1new}
m_0A+m_1+m_2A=r_1^2(\beta^2-\alpha^2)(1-A^2)^{3/2}.
\end{equation}
From \eqref{two-1new} and \eqref{four-1new}, we can conclude that
$$
(\alpha^2-\beta^2)(1-A^2)^{3/2}=0.
$$
Since $A^2=1$ only for singular configurations, which do not occur for equilateral orbits, we can conclude that $|\alpha|=|\beta|$. Consequently equations \eqref{one}, \eqref{two}, 
\eqref{three}, \eqref{four} become equations \eqref{w10}, \eqref{x10}, 
\eqref{y10}, \eqref{z10}, respectively, with $\omega_{01}=\omega_{02}=\omega_{12}=A$.

By Theorem \ref{criterion}, it follows that our candidate solution is necessarily generated from a fixed point configuration. But the only case when an equilateral triangle can be a fixed point of the curved 3-body problem in $\mathbb S^3$ is when the bodies lie on a great circle of a great sphere and all masses are equal.

Alternatively, we could directly prove the equality of the masses under these circumstances by noticing that from equation \eqref{two-1} and the corresponding equation obtained for $i=2$ we respectively obtain the conditions
\begin{equation}\label{last-1}
r_1(m_0A+m_1+m_2A)\sin a_1=0,
\end{equation}
%%%%%%
\begin{equation}\label{last-2}
r_2(m_0A+m_2+m_1A)\sin a_2=0.
\end{equation}
Again, we can assume that $r_1\ne 0, r_2\ne 0, \sin a_1\ne 0$, and $\sin a_2\ne 0$, otherwise we can reach the conclusion that the masses are equal and the solution is generated from a fixed point configuration. But then
from \eqref{last-1} and \eqref{last-2}, we obtain that
$$
(m_1-m_2)(1-A)=0,
$$
which implies that $m_1=m_2$ and, eventually, that all the three masses are equal and the solution is generated from a fixed point configuration. This remark completes the proof.
\end{proof}

%%%%%%%
%%%%%%%

\medskip

\noindent{\bf Acknowledgements.} Florin Diacu acknowledges the partial support of the Discovery Grant 122045 from NSERC of Canada.

%%%%%%%%
%%%%%%%%

\end{document}